\documentclass{amsart}
	\usepackage[utf8]{inputenc}
	\usepackage{pstricks}
	\usepackage{amsfonts}
	\usepackage{amsmath}
	\usepackage{amssymb}
	\usepackage{amsthm}
	\usepackage{mathrsfs}
	\usepackage{enumerate}
	\usepackage{tikz}	
\usepackage{newcent}

	\usetikzlibrary{matrix,arrows}

	\newcommand{\ZZ}{\mathbb{Z}}

	\renewcommand{\bar}{\overline}
		
	\newcommand{\Oo}{\mathscr{O}}	

	\DeclareMathOperator{\cone}{cone}
	\DeclareMathOperator{\Star}{Star}
	
	\DeclareMathOperator{\chara}{char}

	\DeclareMathOperator{\rank}{rank}
	\DeclareMathOperator{\Spec}{Spec}
	
	\DeclareMathOperator{\Proj}{Proj}
	\DeclareMathOperator{\Pic}{Pic}
	
	\DeclareMathOperator{\height}{ht}
	\DeclareMathOperator{\Div}{Div}

	\newcommand{\isom}{\simeq}

	\renewcommand{\phi}{\varphi}
	\renewcommand{\epsilon}{\varepsilon}
	\renewcommand{\tilde}{\widetilde}
	
	\theoremstyle{plain}
	\newtheorem{lemma}{Lemma}

	\newtheorem*{proposition*}{Proposition}
	\newtheorem*{corollary*}{Corollary}
	\newtheorem{theorem}{Theorem}
	\newtheorem*{theorem*}{Theorem}
	\newtheorem{obs}{Observation}
	\theoremstyle{definition}
	
	\theoremstyle{remark}
	\newtheorem*{remark}{Remark}

	\author{Piotr Achinger}
	\title{A characterization of toric varieties in characteristic $p$}
	\numberwithin{equation}{section}
\begin{document}

\maketitle

\begin{abstract}
If $X$ is a smooth toric variety over an algebraically closed field of positive characteristic and $L$ is an invertible sheaf on $X$, it is known that $F_* L$, the push-forward of $L$ along the Frobenius morphism of $X$, splits into a direct sum of invertible sheaves.
We show that this property characterizes smooth projective toric varieties.
\end{abstract}

\section*{Introduction}

If $X$ is a smooth toric variety over an algebraically closed field of positive characteristic and $L$ is an invertible sheaf on $X$, then $F_* L$, the push-forward of $L$ along the Frobenius morphism of $X$, splits into a direct sum of invertible sheaves \cite{Thomsen}, \cite{Bogvad}. 
In this article, we show that this property in fact characterizes smooth projective toric varieties.

\begin{theorem} \label{thm-iff}
Let $X$ be a projective connected scheme over an algebraically closed field $k$ of characteristic $p>0$. Then $X$ is a smooth toric variety if and only if, for every invertible sheaf $L$ on $X$, $F_* L$ is a direct sum of invertible sheaves.
\end{theorem}

Although the ``only if'' part of this result has been known already, we include a short proof for completeness (Theorem \ref{thm-only-if}). The proof of the ``if'' part (Theorem \ref{thm-if}) is a bit more involved: we first reduce to the case when the Picard group of $X$ is free abelian and look at the Cox ring $R$ of $X$. The assumption on $X$ means that $R$ is flat over $R^p$, so $R$ is regular by Kunz's criterion. It follows that $R$ has to be a polynomial ring, so $X$ is a toric variety.

We may expect Theorem 1 to hold more generally for $X$ proper. However, our proof uses the characterization of toric varieties in terms of their Cox rings in the form stated in \cite{KedzierskiWisniewski} which requires $X$ to be projective. It is worth noting that similar characterizations have been obtained \cite[Corollary 4.4]{BerchtoldHausen} for $X$ satisfying the $A_2$ condition: every two points have a common affine neighborhood. These theorems, depending on various results on group actions, have been stated only in characteristic zero, therefore we have to be very careful applying them to actions of tori in positive characteristic.  

\section{Proof of the ``only if'' part}

We provide an a bit shorter and more explicit proof of Thomsen's result \cite{Thomsen}. The proof gives a very short computation even for projective spaces (in which case one can use the Horrocks splitting criterion and the projection formula to prove that the direct image is a direct sum of invertible sheaves). The key point of our approach is to consider Frobenius push-forwards of all invertible sheaves at once.

Let $X$ be a smooth toric variety over an algebraically closed field $k$. By definition, there is a torus $T$ acting on $X$ with an open free orbit, and $X$ is defined by a combinatorial structure called a fan \cite{CoxLittleSchenck}. We denote the number of rays of the fan defining $X$ by $r$. If $\chara k = p > 0$, we have the \emph{absolute} Frobenius morphism $F_a:X\to X$. In fact, as $X$ is naturally defined over $\mathbb{F}_p$, $X$ can be identified with its Frobenius twist $X^{(1)}$ and then $F_a$
can be seen as the quotient by the Frobenius kernel $K_a$ of $T$. In any case, for any integer $\ell>0$ we have the \emph{toric} Frobenius morphism $F_\ell : X\to X$ which corresponds to taking the quotient by the kernel $K_\ell$ of the $\ell$-th power map on $T$. Let $F$ be either $F_\ell$ or $F_a$ (with $\ell := p$ in the latter case) in the following theorem:

\begin{theorem}[{\cite[Theorem 1]{Thomsen}}] \label{thm-only-if}
Let $X$ and $F$ be as in the previous paragraph. Then for every invertible sheaf $L$ on $X$, $F_* L$ is a direct sum of invertible sheaves. Furthermore, we have the following explicit description of the direct summands: Let $D\in\Pic X$. Then
\begin{equation}\label{dec}
	 F_* \Oo_X(D) \isom \bigoplus_{E \in \Pic X} \Oo_X(E)^{\oplus m(E, D)}, 
\end{equation}
where the multiplicity $m(E, D)$ equals the number of points in the cube $\{0, 1, \ldots, \ell-1\}^r\subseteq \ZZ^r \isom \Div_T X$ representing the class $D-\ell E\in\Pic X$ (that is, the number of $T$-invariant divisors in $|D-\ell E|$ with coefficients less than $\ell$).
\end{theorem}

\begin{proof}
Let us first prove the theorem in the case when $X$ is complete, and reduce to this case afterwards. 

First of all, we remark that the push-forward $F_* \Oo_X(D)$ of a $T$-equivariant invertible sheaf $\Oo_X(D)$ is a direct sum of invertible sheaves. Indeed, the kernel $K$ of $F$ on $T$ (equal to $K_a$ or $K_\ell$) is a finite diagonalizable commutative group scheme $\Spec k[\Lambda/\ell\Lambda]$ (where $\Lambda$ is the character lattice of $T$) acting on $F_* \Oo_X(D)$ and the eigensheaves are invertible (the proof is immediately reduced to the case of an affine space and further to the case of the affine line). Since every invertible sheaf on $X$ is equivariant, we get a decomposition as in (\ref{dec}) and we only want to compute the multiplicities.

Observe that $m(E, D)$ depends only on $D - \ell E$: by the projection formula we have $(F_* \Oo_X(D))\otimes \Oo_X(-E) = F_*(\Oo_X(D - \ell E))$, so $m(E, D) = m(0, D-\ell E)$. Denote $m(0, D)$ simply by $m(D)$ and  apply $h^0(-)$ to both sides of (\ref{dec}):
\begin{equation} \label{h0}
	 h^0(D) 
	= \sum_{E \in \Pic X} m(E, D)\cdot h^0(E) = \sum_{E \in \Pic X} m(D - \ell E)\cdot h^0(E). 
\end{equation}

We want to use some generating functions, so we fix a basis $D_1, \ldots, D_\rho$ of $\Pic X$ such that the effective cone lies in the positive orthant and define
\[ S(x) = \sum_{a\in \ZZ^\rho} h^0\left(\sum a_i D_i\right) x^a \quad\text{and}\quad M(x) = \sum_{a\in \ZZ^\rho} m\left(\sum a_i D_i\right) x^a. \]
Then (\ref{h0}) just states that 
\[ S(x) = M(x)\cdot S(x^\ell) \]
where we write $x^\ell$ for $(x_1^\ell, \ldots, x_\rho^\ell)$.
 
Let us compute the series $S$. Consider the map $L : \ZZ^r = \Div_T X \to \Pic X = \ZZ^\rho$ (the first identification being given by the basis of ,,ray'' divisors, the second by $D_1, \ldots, D_\rho$) taking a $T$-divisor to its class. Because $h^0(D)$ equals the number of effective $T$-divisors linearly equivalent to $D$, we get
\[ S(x) = \sum_{b\in \ZZ_{\geq 0}^r} x^{L(b)} = \prod_{i=1}^r \frac{1}{1 - x^{L(e_i)}}, \]
$e_1, \ldots, e_r$ being the basis in $\ZZ^r$. Therefore
\[ M(x) = \frac{S(x)}{S(x^\ell)} = \prod_{i=1}^r \frac{1-x^{\ell L(e_i)}}{1-x^{L(e_i)}}
 = \prod_{i=1}^r \left(1 + x^{L(e_i)} + \ldots + x^{(\ell -1)L(e_i)}\right), \]
hence $m(D)$ is the number of points $p = \sum a_i e_i$ with $0 \leq a_i < \ell$ and $L(p) = D$.

\medskip
We turn to the case $X$ not necessarily complete. By adding extra cones, we embed $i: X\to \bar X$ into a smooth complete toric variety $\bar X$. Using the diagram
\[ 
\begin{tikzpicture}[description/.style={fill=white,inner sep=2pt}]
    \matrix (m) [matrix of math nodes, row sep=1.0em,
    column sep=1.0em, text height=1.5ex, text depth=0.25ex]
    {    &    & 0                       & 0\\
         &    & \ZZ^{\bar r - r} & \ZZ^{\bar r - r} \\
       0 & M  & \Div_T \bar X           & \Pic \bar X  & 0 \\
       0 & M  & \Div_T X           & \Pic X & 0 \\ 
         &    & 0                         & 0 \\};
    \path[double]
    (m-2-3) edge[double distance=2pt] node[auto] {$ $} (m-2-4);

    \path[->,font=\scriptsize]
    (m-3-1) edge node[auto] {$ $} (m-3-2)
    (m-3-2) edge node[auto] {$ $} (m-3-3)
    (m-3-3) edge node[auto] {$ $} (m-3-4)
    (m-3-4) edge node[auto] {$ $} (m-3-5);
    \path[->,font=\scriptsize]
    (m-4-1) edge node[auto] {$ $} (m-4-2)
    (m-4-2) edge node[auto] {$ $} (m-4-3)
    (m-4-3) edge node[auto] {$ $} (m-4-4)
    (m-4-4) edge node[auto] {$ $} (m-4-5);
    \path[double]
    (m-3-2) edge[double distance=2pt] node[auto] {$ $} (m-4-2);
    \path[->,font=\scriptsize]
    (m-1-3) edge node[auto] {$ $} (m-2-3)
    (m-2-3) edge node[auto] {$ $} (m-3-3)
    (m-3-3) edge node[auto] {$ $} (m-4-3)
    (m-4-3) edge node[auto] {$ $} (m-5-3);
    \path[->,font=\scriptsize]
    (m-1-4) edge node[auto] {$ $} (m-2-4)
    (m-2-4) edge node[auto] {$ $} (m-3-4)
    (m-3-4) edge node[auto] {$ $} (m-4-4)
    (m-4-4) edge node[auto] {$ $} (m-5-4);
\end{tikzpicture}
\]
(where $M$ is the character lattice of $T$) we see that if the theorem holds for $\bar X$, then it also holds for $\tilde X$ and for $X$ (the variety $X$ being normal).
\end{proof}

\subsection{Vanishing of cohomology of invertible sheaves -- a result of Borisov and Hua}

Using the above description of Frobenius push-forwards of invertible sheaves on toric varieties, one can easily calculate the ranks of cohomology groups of invertible sheaves (\cite{BorisovHua}). The idea is as follows: since the Frobenius maps are affine, we have $H^i(X, L) = H^i(X, F_{\ell*} L)$, and using Theorem \ref{thm-iff} we can see that for $\ell$ large enough $F_{\ell*} L$ is a direct sum of invertible sheaves from a fixed \emph{finite} subset of $\Pic X$. Before we state the result, we need some definitions.

For a subset $I$ of the set $\Delta(1)$ of boundary divisors of $X$ (rays of $\Delta$), let $D_I = \bigcup_{i\in I} D_i$ and $L_I = \Oo_X(-D_I)$ (the ideal sheaf of $D_i$). Let $K \subseteq \Pic(X)_\mathbb{Q}$ be the convex hull of the classes $[L_I]$ for all $I\subseteq \Delta(1)$. We call a subset $I\subseteq \Delta(1)$ \emph{extremal} if $[L_I]$ is a vertex of $K$. 
Let 
\[ B_k = \{ I\subseteq\Delta(1)\, |\, I\text{ is extremal and }H^k(X, L_I)\neq 0 \}, \quad k=0, 1, \ldots, \dim X. \] 
For $I\subseteq \Delta(1)$, let $C_I$ be the translated cone
\[ C_I = [L_I] + \cone([L_I] - [L_J]\, J\subseteq \Delta(1)) = 
[L_I] + \cone([L_I] - [L_J]\, J\subseteq \Delta(1)\text{ is extremal}). \]
and finally let
\[ A_k = \{ [L]\in\Pic X \, | \, H^k(X, L)\neq 0\}, \quad k=0, 1, \ldots, \dim X. \]

\begin{theorem}[{\cite[Section 4]{BorisovHua}}]\label{thm-bh} We have $A_k = \bigcup_{I\in B_k} C_I$. 
\end{theorem}

\begin{proof} Theorem \ref{thm-iff} states that $\Oo_X(E)$ is a direct summand of $F_{\ell*} \Oo_X(D)$ if and only if 
\begin{equation}\label{eqn-bla} [E] \in \frac{\ell-1}{\ell}K + \frac{1}{\ell}[D]. \end{equation}
In particular, given $D$, we have that all $E$ in $F_{\ell*} \Oo_X(D)$ lie in $K$ for $\ell\gg 0$. Since $F_{\ell}$ is affine, we have $[D]\in A_k$ if and only if $[E]\in A_k$ for some $E$ such that $\Oo_X(E)$ appears in $F_{\ell*} \Oo_X(D)$. 

Recall that if $P$ is a convex polytope and $v$ is a vertex of $P$, the \emph{star} of $v$ (denoted $\Star_P(v)$) is the open subset of $P$ consisting of interiors of all faces of $P$ containing $v$. Note that
\[ \Star_P(v) = \bigcup_{\ell\gg 0} \left(v + \frac{\ell-1}{\ell}(P-v)\right).\]
The above equality together with \eqref{eqn-bla} show that if $I\subseteq \Delta(1)$ is extremal and $J\subseteq \Delta(1)$ is such that $[L_J] \in \Star_K([L_I])$ then $L_J$ is a direct summand of $F_{\ell*} L_I$ for $\ell\gg 0$. 

We claim that if $J\subseteq \Delta(1)$ is \emph{not} extremal then all the cohomology groups of $L_J$ vanish. For this, we choose an extremal $I\subseteq \Delta(1)$ such that $[L_J]\in \Star_K([L_I])$ and and $\ell>0$ such that $L_J$ is a direct summand of $F_{\ell*} L_I$. Since $L_I$ is a direct summand of $F_{\ell*} L_I$ as well and the push-forward does not change cohomology, we see that all the cohomology groups of $L_J$ have to be zero.

Fix a $D\in \Pic X$ and choose an $\ell\gg 0$ such that $(K+ \frac{1}{\ell}D)\cap \Pic X \subseteq K \cap \Pic X$. Then 
\begin{align*}
[D]\in A_k & \text{ iff }H^k(X, \Oo_X(D)) \neq 0 \\
            &  \text{ iff }H^k(X, \Oo_X(E))\neq 0 \text{ for some }E\text{  such that }[E] \in \frac{\ell-1}{\ell}K + \frac{1}{\ell}[D] \\ 
            & \text{ iff }H^k(X, L_I) \neq 0\text{ for some } I\subseteq \Delta(1)\text{ s.t. }[L_I] \subseteq \frac{\ell-1}{\ell}K + \frac{1}{\ell}[D] \\
            & \text{ iff }H^k(X, L_I) \neq 0\text{ for some extremal } I\subseteq \Delta(1)\text{ s.t. }[L_I] \subseteq \frac{\ell-1}{\ell}K + \frac{1}{\ell}[D] \\
            &\text{ iff }H^k(X, L_I) \neq 0\text{ for some extremal } I \text{ s.t.}[D] \in [L_I] + (\ell-1)([L_I] - K).
\end{align*} 
To finish the proof, it suffices to note that
\[  C_I = \bigcup_{\ell\gg 0} \left([L_I] + (\ell-1)([L_I] - K)\right). \qedhere \]
\end{proof}

\begin{remark} With some additional care, we can recover the dimensions of cohomology groups $h^k(D) := \dim H^k(X, \Oo_X(D))$. Recall that $B_k$ is set of all $I\in\Delta(1)$ such that $H^k(X, L_I)\neq 0$. Note that any $I\in B_k$ is extremal and that for such $I$, the number $\mu(D, I)$ of copies of $L_I$ in $F_{\ell*} \Oo_X(D)$ is independent of $\ell$ for $\ell\gg 0$. We have
\[ h^k(D) = \sum_{I\in B_k} h^k(L_I)\cdot \mu(D, I) \]
\[ = \sum_{I\in B_k} h^k(L_I)\cdot (\text{coeff. of }x^{D-\ell L_I}\text{ in }\prod_i(1+x^{L(e_i)} + \ldots + x^{(\ell-1)L(e_i)})). \] 
As $L_I$ are the ideal sheaves of the reduced boundary divisors $D_I$ and $h^i(\Oo_X) = 0$ for $i>0$, we have $h^i(L_I) = h^{i-1}(\Oo_{D_I})$ for $i>1$,  $h^1(L_I) = h^0(\Oo_{D_I}) - 1$ for $I\neq\emptyset$, $h^0(L_I) = 0$ for $I\neq \emptyset$. The cohomology groups of $\Oo_{D_i}$ can be computed using an explicit simplicial complex.
\end{remark}

\section{Proof of the ``if'' part}

Let $k$ be an algebraically closed field of characteristic $p>0$. 

\begin{theorem} \label{thm-if} 
Let $X$ be projective and connected scheme over $k$. Suppose that $X$ satisfies the following condition: for every invertible sheaf $L$ on $X$, the push-forward $F_* L$ of $L$ along the (absolute) Frobenius morphism of $X$ is a direct sum of invertible sheaves. Then $X$ is a smooth toric variety.
\end{theorem}

The following two lemmas show that $\Pic X$ is finitely generated and $p$-torsion free. Recall that a scheme $X$ of characteristic $p$ is called \emph{Frobenius split} if the canonical map $F^\sharp : \Oo_X\to F_* \Oo_X$ is a split monomorphism.

\begin{lemma} \label{sm-fs} 
Let $X$ be proper scheme over $k$. Suppose that $F_* \Oo_X$ is a direct sum of invertible sheaves. Then $X$ is smooth and Frobenius split. 
\end{lemma}

\begin{proof}
Since $F_* \Oo_X$ is a direct sum of invertible sheaves, it is locally free, so $F$ is a finite flat morphism. We check that $X$ is reduced: if $f\in \Oo_X$ and $f^p=0$, then $f$ acts trivially on the locally free $\Oo_X$-module $F_* \Oo_X$, so $f=0$. Therefore $X$ is regular by Kunz's criterion \cite{Kunz}: \emph{a reduced local ring $R$ of characteristic $p>0$ is regular if and only if it is flat over $R^p$}.

To prove the second assertion, first note that the map $F^\sharp: \Oo_X\to F_* \Oo_X$ is locally split (because smooth affine varieties are Frobenius split: see e.g. \cite[Proposition 1.1.6]{BrionKumar}). Let $F_* \Oo_X = \bigoplus_i L_i$ where the $L_i$ are invertible sheaves. We have $\sum_i \dim_k H^0(X, L_i) = \dim_k H^0(X, F_* \Oo_X) = \dim_k H^0(X, \Oo_X) = 1$, so we can assume that $H^0(X, L_0) \isom k$ and $H^0(X, L_i) = 0$ for $i\neq 0$. It follows that $F^\sharp$ factors through the inclusion $L_0 \hookrightarrow F_* \Oo_X$. The induced map $\Oo_X\to L_0$ is a locally split homomorphism between invertible sheaves, hence an isomorphism. 
\end{proof}

\begin{lemma} \label{noptorsion} 
Let $X$ be as in Theorem \ref{thm-if}. Then $(\Pic^0 X)_{red} = 0$ and $\Pic X$ has no $p$-torsion. 
\end{lemma}

\begin{proof} 
Since $X$ is smooth and Frobenius split by Lemma \ref{sm-fs}, by \cite[Theorem 8.2]{Joshi} we know that $(\Pic^0 X)_{red}$ is Frobenius split as well therefore \emph{ordinary} \cite[Remark 1.3.9 (ii)]{BrionKumar}. Recall that an abelian variety $A$ is ordinary if $A[p]\isom (\ZZ/p\ZZ)^{\dim A}$. Therefore the triviality of $(\Pic^0 X)_{red}$ will follow from the second assertion.

Suppose now that there exists a $p$-torsion invertible sheaf $M$ on $X$. We claim that this implies that for every invertible sheaf $L$ on $X$, $\chi(L):= \sum_{i\geq 0} (-1)^i \dim_k H^i(X, L)$ is divisible by $p$. 
Let $m(L', L)$ denote the multiplicity of $L'$ as a direct summand of $F_* L$. Then $m(L'\otimes M, L) = m(L', L)$: since $M$ is torsion, it suffices to prove that $m(L'\otimes M, L) \geq m(L', L)$ and if $L'^{\oplus m(L', L)}$ is a direct summand of $F_* L$, then 
$(L'\otimes M)^{\oplus m(L', L)}$ is a direct summand of $(F_* L)\otimes M = F_* (L\otimes F^* M) = F_* (L\otimes M^{p}) = F_* L$. Since the Euler characteristic of an invertible sheaf depends only on its numerical class (as seen, for example, using the Hirzebruch-Riemann-Roch formula) and $M$ is numerically trivial, we have $\chi(L') = \chi(L'\otimes M)$.
So $\chi(F_* L)$ is divisible by $p$, but $\chi(F_* L) = \chi(L)$ (because $F$ is affine). This proves our claim.

Now since for any invertible sheaf $L$, $p$ divides $\chi(L)$, we easily show by induction on $r$ that $p^r$ divides $\chi(F^r_* L) = \chi(L)$. Therefore $\chi(L)=0$ for all invertible sheaves $L$. This contradicts the fact that for $L$ ample and a large enough $m$, $H^i(X, L^m)=0$ for $i>0$ and $H^0(X, L^m)\neq 0$.
\end{proof}

\begin{remark}
The last step of the above proof requires $X$ to be projective. Trying to generalize Theorem \ref{thm-if} to proper or proper and $A_2$ varieties would require producing invertible sheaves of nonzero Euler characteristic. We do not know whether every smooth proper variety possesses an invertible sheaf with nonzero Euler characteristic. If $X$ is proper and $A_2$, $X$ can be embedded in a toric variety \cite{Wlodarczyk}, and one could restrict invertible sheaves from the ambient variety and use Hirzebruch-Riemann-Roch and intersection theory on the ambient variety to prove that the Euler characteristic would be nonzero. This requires showing that if $X$ is a closed irreducible smooth subvariety of a toric variety then the cohomology class of $X$ is nonzero. We do not know if this holds in general. 
\end{remark}

Let $X$ be as in Theorem \ref{thm-if}. By Lemma \ref{noptorsion} and the general structure of Picard groups of smooth projective varieties \cite{Kleiman}, we see that $\Pic X\isom \Lambda' \oplus T$ where $\Lambda'$ is free of finite rank and $T$ is finite of order prime to $p$. By the next Lemma (which is easy and well-known), we can then find a finite \'etale cover $\pi:Y\to X$ with Galois group $T$ such that the kernel of the induced map $\pi^*: \Pic X\to \Pic Y$ equals $T$. Let $\Lambda = \pi^* (\Pic X) \subseteq \Pic Y$ (note that $\Lambda\isom\Lambda'$, so $\Lambda$ is free of finite rank).

\begin{lemma}
Let $X$ be a connected proper algebraic variety, $H$ a finite subgroup of $\Pic X$ of order prime to the characteristic. Then there exists a finite \'etale map $f:Y\to X$ with $Y$ connected such that the kernel of the induced map $f^* : \Pic X\to \Pic Y$ equals $H$.
\end{lemma}

\begin{proof} Presenting $H$ as a direct sum of cyclic groups, we easily reduce to the case $H$ cyclic, generated by the class of an $n$-torsion invertible sheaf $L$. In this case, we construct $Y$ as the standard cyclic cover $Y = \Spec_X A$ where $A= \bigoplus_{i=0}^{n-1} L^i$ with an algebra structure fixed by choosing an isomorphism $L^n \isom \Oo_X$. The projection map $f:Y\to X$ is clearly finite \'etale, so we only need to check that the kernel of the pull-back map induced by $f$ on the Picard groups consists of the $L^j$. Under the usual correspondence between coherent sheaves on $Y$ and coherent sheaves of $A$-modules on $X$, the pull-back $f^* M$ corresponds to $\bigoplus_{i=0}^{n-1} L^i \otimes M$, which is a free $A$-module for $M \isom L^j$ for some $j$. On the other hand, if $\bigoplus_{i=0}^{n-1} L^i \otimes M$ is a free $A$-module, then 
\[ \bigoplus_{i=0}^{n-1} L^i \otimes M \isom A = \bigoplus_{i=0}^{n-1} L^i \]
as $\Oo_X$-modules. By the Krull-Schmidt theorem \cite[Theorem 3]{Atiyah}, as invertible sheaves are indecomposable $\Oo_X$-modules, the summands above have to be the same up to permutation, hence $M \isom L^j$ for some $j$.
\end{proof}

Let $L$ be an invertible sheaf on $Y$ whose isomorphism class belongs to $\Lambda$. We claim that $F_* L$ is a direct sum of invertible sheaves whose isomorphism classes also belong to $\Lambda$. We have $L\isom \pi^* M$ for some invertible sheaf $M$ on $X$. By \cite[XIV=XV \S{}1 $n^\circ$ 2, Pr. 2(c)]{SGA5}, the diagram 
\[
\begin{tikzpicture}[description/.style={fill=white,inner sep=2pt}]
    \matrix (m) [matrix of math nodes, row sep=2.0em,
    column sep=2.5em, text height=1.5ex, text depth=0.25ex]
    {     Y  & Y \\
	X  & X \\ };
    \path[->,font=\scriptsize]
    (m-1-1) edge node[auto] {$ F_Y $} (m-1-2)
    (m-1-2) edge node[auto] {$ \pi $} (m-2-2);
    \path[->,font=\scriptsize]
    (m-1-1) edge node[auto] {$ \pi $} (m-2-1)
    (m-2-1) edge node[auto] {$ F_X $} (m-2-2);
\end{tikzpicture}
\]
is a cartesian diagram of flat morphisms. By flat base change \cite[Proposition 9.3]{Hartshorne}, $F_{Y*} L = F_{Y*} \pi^* M = \pi^* F_{X*} M$ is a direct sum of invertible sheaves. 

We denote by $R = \bigoplus_{\lambda \in \Lambda} R_\lambda$ the global  section ring of $(Y, \Lambda)$, that is
\[ R = \bigoplus_{L\in \Lambda} H^0(Y, L) \]
where the ring structure is fixed by arbitrarily choosing a collection of invertible sheaves whose classes form a basis of $\Lambda$.

\begin{lemma} \label{regular} 
Under the above assumptions
\begin{enumerate}
	\item $R$ is a flat $R^p$-algebra,
	\item $R$ is a finitely generated $k$-algebra,
	\item $R$ is regular.
\end{enumerate}
\end{lemma}  

\begin{proof}
(1) For $\mu\in \Lambda/p\Lambda$, let $\mu'$ be an arbitrary representative of $\mu$ in $\Lambda$. Then $R$ decomposes as an $R^p$-module:
\[ R = \bigoplus_{\mu\in\Lambda/p\Lambda} M_\mu \quad\text{where}\quad M_\mu = \bigoplus_{\lambda\in\Lambda} H^0(Y, L_{\mu' + p\lambda}) \]
(here $L_\lambda$ is a chosen invertible sheaf representative of $\lambda\in\Lambda\subseteq \Pic Y$). Using the projection formula, the summands can be presented as:
\[ M_\mu = \bigoplus_{\lambda\in\Lambda} H^0(Y, (F_*L_{\mu'})\otimes L_\lambda)) = \bigoplus_{i=1}^{p^d} \bigoplus_{\lambda\in\Lambda} H^0(Y, L_{\lambda_i + \lambda}) = \bigoplus_{i=1}^{p^d} R^p(\lambda_i), \]
where $d=\dim Y$ (so $p^d = \deg F_Y$) and $F_* L_{\mu'} = \bigoplus L_{\lambda_i}$. Therefore all $M_\mu$ are free, and hence also $R$ is a free $R^p$-module. 

(2) Note the above reasoning exhibits $R$ as a \emph{graded} $R^p$-module (though not in a canonical way). We can therefore choose a system $f_i \in R_{\lambda_i}$ ($i=1,\ldots, k$) of homogeneous generators of $R$ over $R^p$. I claim that the $f_i$ generate $R$ as a $k$-algebra. Let $f\in R_\lambda$ be a homogeneous element of $R$. By assumption, we can write $f = \sum g_i^p f_i$ for some $g_i$ (which can be assumed to be homogeneous). We can again write $g_i = \sum h_{ij}^p g_j$ and so on, and this process cannot go on forever since $\Lambda$ is not infinitely $p$-divisible\footnote{Note the example $R = k[t^{1/p^\infty}]$ which is flat over $R^p$ but not Noetherian}.

(3) follows from (1)-(2) by Kunz's criterion \cite{Kunz}.
\end{proof}

\begin{lemma} \label{polynomial} 
Let $\Lambda$ be a finitely generated free abelian group and let $R = \bigoplus_{\lambda \in\Lambda} R_\lambda$ be a $\Lambda$-graded domain. Suppose that $m := \bigoplus_{\lambda\neq 0} R_\lambda$ is an ideal of $R$ and that $k:= R_0 = R/m$ is a field. If $R$ is regular, it is isomorphic, as a $\Lambda$-graded ring, to $k[t_1, \ldots, t_r]$ with the $t_i$ homogeneous.
\end{lemma}

\begin{proof} 
Since $R$ is regular, $r:=\dim R$ is finite and the maximal ideal $m$ is generated by $r$ elements $(x_1, \ldots, x_r)$. We can assume that $x_i$ is homogeneous of degree $\lambda_i$ (pick a homogeneous basis of the $r$-dimensional vector space $m/m^2$ and take homogeneous parts of lifts). 
Consider the map $\psi: k[t_1, \ldots, t_r] \to R$ sending $k$ to $R_0\isom k$ and $t_i$ to $x_i$. It is easily seen to be surjective by induction on the multidegree. It is also injective since $r = \dim k[t_1, \ldots, t_r] = \dim R + \height \ker \psi = r + \height \ker \psi$.  
\end{proof}

\begin{proof}[Proof of Theorem \ref{thm-if}]
Let $\pi:Y\to X$ be the finite \'etale cover constructed previously, with $\Lambda = \pi^* \Pic X \subseteq \Pic Y$. Combining Lemmas \ref{regular} and \ref{polynomial} we see that the section ring of $(Y, \Lambda)$ is a polynomial ring. Because $\Lambda$ contains an ample class (take a pull-back of an ample class on $X$), $Y$ is toric by \cite[Theorem 1.5]{KedzierskiWisniewski}. Now $\pi:Y\to X$ is \'etale, $Y$ is rational, so $X$ is separably rationally connected, hence simply connected by \cite[Corollaire 3.6]{Debarre}, so $\pi$ is an isomorphism and $X$ is toric.
\end{proof}

\begin{remark}
Theorem 1.5 from \cite{KedzierskiWisniewski} assumes that the characteristic of $k$ is zero. This is however not necessary. Let us briefly outline the proof of this result in arbitrary characteristic.

As noted by Alper \cite[Remark 1.1]{Alper}, Luna's \'etale slice theorem \cite{Luna} works without change in case of linearly reductive groups, that is, it works for tori in characteristic $p$. Therefore the proof of \cite[Theorem 5.1]{KraftPopov} works without changes other than obvious simplifications as well. It follows that if $R$ is a polynomial ring, we can assume that $R = k[x_1, \ldots, x_r]$ with the $x_i$ homogeneous. Then $X = \Proj \bigoplus_{n\geq 0} R_{n\lambda}$, where $\lambda$ is an ample class, is a toric variety determined by the combinatorial data of the choice of $\lambda$ and the gradings of $x_i$. 
\end{remark}

\section{End remarks}

Let $X$ be a smooth projective variety over an algebraically closed field $k$ of characteristic $p>0$. We denote the absolute Frobenius morphism of $X$ by $F$. We can study the following properties of $X$:
\begin{itemize}
\item[(A)] For any invertible sheaf $L$ on $X$, $F_* L$ is a direct sum of invertible sheaves.
\item[(B)] $F_* \Oo_X$ is a direct sum of invertible sheaves.
\item[(C)] $F^s_* \Oo_X$ is a direct sum of invertible sheaves for $s\gg 0$.
\end{itemize}
Our Theorem \ref{thm-iff} says that property (A) holds if and only if $X$ is a toric variety.

\begin{obs} If $X$ satisfies (C), then $X$ need not necessarily be toric. \end{obs}

For example, if $X$ is an ordinary abelian variety, then it is Frobenius split, i.e., $\Oo_X$ is a direct summand of $F_* \Oo_X$ and hence of all $F^s_* \Oo_X$. If $L$ is a $p^s$-torsion invertible sheaf on $X$, then tensoring $\Oo_X \to F^s_* \Oo_X \to \Oo_X$ by $L$ we get that $L$ is a direct summand of $(F^s_* \Oo_X)\otimes L = F^s_*(\Oo_X\otimes F^{s*} L) = F^s_*(L^{p^s}) = F^s_* \Oo_X$, that is, $F^s_* \Oo_X$ contains as a direct summand the direct sum of all $p^s$-torsion invertible sheaves. But since $X$ is ordinary, it has exactly $p^{s\cdot\dim X} = \rank F^s_* \Oo_X$ such invertible sheaves, hence 
\[ F^s_* \Oo_X \isom \bigoplus_{L\in \Pic X[p^s]} L. \] 

Another interesting fact is that (B) can hold for special non-toric Fano varieties:

\begin{obs}[{\cite[Theorem 2]{Achinger}}] For $p>2$ and $X$ a smooth quadric of even dimension $2n$, $F^s_* \Oo_X$ is a direct sum of invertible sheaves if and only if $p|n$ and $s=1$. \end{obs}

It would be interesting to have a structure theorem for all varieties satisfying (B) or (C) -- for example, that in the case (C) the Albanese map $f:X\to A$ makes $X$ into a toric fibration over an ordinary abelian variety (that is, $X$ is an \emph{equivariant embedding of an ordinary semiabelian variety}). 

\medskip
\noindent {\bf Acknowledgements. } 
The author would like to thank Morgan Brown, Nathan Ilten, Mateusz Michałek, Nicolas Perrin, Arthur Ogus, Martin Olsson and Jarosław Wisniewski for valuable suggestions. 

\bibliographystyle{amsalpha} 
\providecommand{\bysame}{\leavevmode\hbox to3em{\hrulefill}\thinspace}
\providecommand{\MR}{\relax\ifhmode\unskip\space\fi MR }
\providecommand{\MRhref}[2]{%
  \href{http://www.ams.org/mathscinet-getitem?mr=#1}{#2}
}
\providecommand{\href}[2]{#2}

\medskip
\noindent {\sc Piotr Achinger \\
Department of Mathematics \\
University of California, Berkeley \\
Berkeley, CA 94720, USA} \\

\noindent {\sc E-mail:} \texttt{achinger@math.berkeley.edu}

\end{document}